\documentclass[12pt]{article}

\usepackage{amssymb}%
\usepackage{amsmath}%
\usepackage{amsthm}%
\usepackage{mathrsfs}

\usepackage{xcolor}%

\usepackage{tikz-cd}

\allowdisplaybreaks%

{\theoremstyle{plain}%
 \newtheorem{theorem}{Theorem}

}
{\theoremstyle{remark}
\newtheorem{remark}{Remark}
}
{\theoremstyle{definition}
\newtheorem{definition}{Definition}
\newtheorem{example}{Example}
}

\begin{document}
 
\begin{center}

 A generalization of immanants based on partition algebra characters

 \ 

John M.\ Campbell

\end{center}

\begin{abstract}
We introduce a generalization of immanants of matrices, using partition algebra characters in place of symmetric group characters. We prove that our 
 immanant-like function on square matrices, which we refer to as the \emph{recombinant}, agrees with the usual definition for immanants for the special 
 case whereby the vacillating tableaux associated with the irreducible characters correspond, according to the Bratteli diagram for partition algebra 
 representations, to the integer partition shapes for symmetric group characters. In contrast to previously studied variants and generalizations of 
 immanants, as in Temperley--Lieb immanants and $f$-immanants, the sum that we use to define recombinants is indexed by a full set of partition 
 diagrams, as opposed to permutations.
\end{abstract}


\noindent  {\footnotesize immanant, partition algebra, character, irreducible representation}

\section{Introduction}
 The concept of the \emph{immanant} of a matrix was introduced in a seminal 1934 article by Littlewood and Richardson \cite{LittlewoodRichardson1934}. 
 As suggested by Littlewood and Richardson \cite{LittlewoodRichardson1934}, by generalizing determinants and permanents of matrices using symmetric 
 group characters, this provides a way of unifying disparate areas of combinatorial analysis, linear algebra, and representation theory. Since partition 
 algebras are such natural extensions of symmetric group algebras \cite{HalversonRam2005}, this leads us to consider how immanants of matrices may be 
 generalized using partition algebra characters. This forms the main purpose of our article, in which we introduce the concept of the \emph{recombinant} of 
 a matrix. This gives us a generalization of immanants that is separate from the concept of an $f$-immanant. 

 Given an $n \times n$ matrix 
\begin{equation}\label{generalnxn}
 A = \left( a_{i, j} \right)_{n \times n} = \left( \begin{matrix} 
 a_{1, 1} & a_{1, 2} & \cdots & a_{1, n} \\ 
 a_{2, 1} & a_{2, 2} & \cdots & a_{2, n} \\ 
 \vdots & \vdots & \ddots & \vdots \\ 
 a_{n, 1} & a_{n, 2} & \cdots & a_{n, n} 
 \end{matrix} \right), 
\end{equation}
 the Leibniz identity for determinants is as below: 
\begin{equation}\label{Leibniz}
 \text{det}(A) = \sum_{\sigma \in S_{n}} \left( \text{sgn}(\sigma) \prod_{i=1}^{n} a_{i, \sigma_{i}} \right), 
\end{equation}
 letting $S_{n}$ denotes the group of all permutations of $\{ 1, 2, \ldots, n \}$. 
 The \emph{permanent} of \eqref{generalnxn} is defined by replacing the sign function in 
 \eqref{Leibniz} as below: 
\begin{equation}\label{permdefinition}
 \text{perm}(A) = \sum_{\sigma \in S_{n}} \prod_{i = 1}^{n} a_{i, \sigma_{i}}. 
\end{equation}
 The matrix functions in \eqref{Leibniz} and \eqref{permdefinition} 
 are special cases of the immanant function defined in \cite{LittlewoodRichardson1934} and as below. 

 An integer partition is a finite tuple $\lambda$ of non-increasing natural numbers. If the sum of all of the entries of $\lambda$ is a natural number $n$, then 
 $\lambda$ is said to be a partition of $n$, and this is denoted as $\lambda \vdash n$. For $\lambda \vdash n$, we may let $\chi_{S_n}^{\lambda}$ 
 be the irreducible character that is of the symmetric group $S_n$ and that corresponds to $\lambda$. The \emph{immanant} $\text{Imm}^{\lambda}$ of 
 \eqref{generalnxn} may be defined so that: 
\begin{equation}\label{generalImm}
 \text{Imm}^{\lambda}(A) = 
 \sum_{\sigma \in S_{n}} \chi^{\lambda}_{S_{n}}(\sigma) \prod_{i = 1}^{n} a_{i, \sigma_{i}}. 
\end{equation}
 We find that the $\lambda = (1^{n})$ case of \eqref{generalImm} agrees with \eqref{Leibniz} and the $\lambda = (n)$ case of \eqref{generalImm} 
 agrees with \eqref{permdefinition}. The purpose of this article is to generalize \eqref{Leibniz}, \eqref{permdefinition}, and \eqref{generalImm} using 
 partition algebra characters, as opposed to symmetric group characters. 

 Immanants are of interest within many different areas of advanced linear algebra; see 
 \cite{CoelhoDuffner2001,DuffnerGutermanSpiridonov2021,Gamas2000,GroneMerris1987,Hartmann1985,Heyfron1991,James1994,Kostant1995,Pate1999,Tabata2016}, 
 for example, and many related references. 
 The definition of immanants in terms of the irreducible characters of the symmetric group naturally lends itself
 to applications related to many different areas of algebraic combinatorics; 
 for example, see \cite{ClearmanSheltonSkandera2011,GouldenJackson1992,Greene1992,Haiman1993,Konvalinka2010,StanleyStembridge1993}
 and many similar references. 
 The foregoing considerations reflect the interdisciplinary nature about immanants and motivate our generalization of immanants. 

 Let $V$ denote an $r$-dimensional vector space. Let the general linear group $\text{GL}_{r}(\mathbb{C})$ act on the tensor space $V^{\otimes n}$ 
 diagonally. By taking $S_{r}$ as a subgroup of $\text{GL}_{r}(\mathbb{C})$ and restricting the action of $\text{GL}_{r}(\mathbb{C})$ to permutation 
 matrices, partition algebras may be defined via the centralizer algebra 
\begin{equation}\label{maincentralizer}
 P_{n}(r) \cong \text{End}_{S_{r}}\left(V^{\otimes n}\right), 
\end{equation}
 and the study of partition algebras had arisen within the field of statistical mechanics via 
 the centralizer algebra in \eqref{maincentralizer}, with reference to the work of Jones
 \cite{Jones1994} and Martin \cite{Martin2000,Martin1991,Martin1994,Martin1996}. 
 This again speaks to the interdisciplinary interest surrounding our generalization of immanants 
 via partition algebra characters. 

\subsection{Preliminaries}\label{Preliminaries}
 Our notation concerning partition algebras is mainly borrowed from Halverson's article on the character theory for partition algebras \cite{Halverson2001}. 
 For the sake of breivty, we assume familiarity with partition diagrams and the multiplication of partition diagrams, referring to \cite{Halverson2001} for 
 details. We let $P_{n}(r)$ denote the $\mathbb{C}$-span of all order-$n$ partition diagrams, and we endow this space with the multiplicative 
 operation specified in \cite{Halverson2001}. Structures of this form are referred to as \emph{partition algebras}. 
 We find that the symmetric group algebra 
 of order $n$ spanned by $\mathbb{C}$ is naturally a subalgebra, by taking 
 the span of partition diagrams of order $n$ with $n$ components with exactly one vertex in the upper row 
 and exactly one vertex in the lower row. 

 For integer partitions $\lambda$ and $\mu$, if $\mu_{i} \leq \lambda_{i}$ for all $i$, then $ \lambda / \mu$ denotes the skew shape obtained by 
 removing $\mu$ from $\lambda$. We adopt the convention whereby the upper nodes of a partition diagram of order 
 $n$ are labeled with $1$, $2$, $\ldots$, $n$ and whereby the lower nodes of this diagram are labeled 
 with $1'$, $2'$, $\ldots$, $n'$. We then let $P_{n-1}(x)$ be embedded in $P_{n}(x)$ by adding vertices labeled with $n$ and $n'$ and by letting 
 these vertices be adjacent. From the branching rules subject to the restriction from $P_{n}(r)$ to $P_{n-1}(r)$, and with the use of double centralizer 
 theory via \eqref{maincentralizer}, it can be shown that the irreducible representations of $P_{n}(r)$ are in bijection with 
\begin{equation}\label{irreducibleindexed}
 \widehat{P_{n}(r)} = \{ \lambda \vdash r \ : \ |\lambda^{\ast}| \leq n \}, 
\end{equation}
 where $\lambda^{\ast} = \lambda / (\lambda_{1})$. 

 We let $M^{\lambda}$ denote the irreducible representation of $P_{n}(r)$ indexed by $\lambda \in \widehat{P_{n}(r)}$. Following 
 \cite{Halverson2001}, we establish a bijection between \eqref{irreducibleindexed} and the set $ \widehat{P_{n}} $ consisting of all expressions of the form 
 $\lambda^{\ast}$ in \eqref{irreducibleindexed}, i.e., by mapping $\lambda$ to $\lambda^{\ast}$ and, conversely, by adding a row to $\lambda^{\ast}$ 
 appropriately. For $\mu \in \widehat{P_{n}}$, we may let $\chi_{P_{n}(x)}^{\lambda}$ denote the irreducible character of $P_{n}(x)$ corresponding 
 to $M^{\lambda}$. 

 A basic result in the representation theory of groups is given by how characters are constant on conjugacy classes. Halverson \cite{Halverson2001} 
 introduced a procedure for collecting partition diagrams so as to form analogues of conjugacy classes, referring to \cite{Halverson2001} for details. For a 
 diagram $d$, we let $d_{\mu}$ denote the conjugacy class representative such that $\chi(d) = \chi(d_{\mu})$ for a given partition algebra character $\chi$. 
 
\section{A generalization of immanants}
 For a permutation $p$ of order $n$ that we denote as a function 
\begin{equation}\label{pasfunction}
 p\colon \{ 1, 2, \ldots, n \} \to \{ 1, 2, \ldots, n \}, 
\end{equation}
 we identify this permutation with the partition diagram corresponding to $\{ \{ 1, (p(1))' \}$, $ \{ 2, (p(2))' \}$, $ \ldots$, 
 $ \{ n, (p(n))' \} \}$. We then consider this partition diagram 
 as being associated with the product 
\begin{equation}\label{associatedproduct}
 \prod_{i=1}^{n} a_{i, p(i)}, 
\end{equation}
 for the matrix $A$ in \eqref{generalnxn}, 
 and with regard to the summand in \eqref{generalImm}. 
 So, this raises the question as to what would be appropriate as an analogue of the product in \eqref{associatedproduct}, 
 for an \emph{arbitrary} partition diagram. This leads us toward the following. 

\begin{definition}\label{20230999828818177217707A7M1A}
 For the $n \times n$ matrix in \eqref{generalnxn}, we let the product $ \prod_{d} a_{i, j}$ or $ \prod_{d} A$ be defined in the following manner. 
 If $d$ is of propagation number $0$, then we let the expression $\prod_{d} a_{i, j}$ vanish. If $d$ is of a positive propagation number, let $B$ be a 
 component of $d$ that is propagating. We then form the product of all expressions of the form $a_{i, j}$ such that $i$ is in $B$ and $j'$ is in $B$. Let $ 
 \Pi_{B}$ denote this product we have defined using the component $B$. We then define $\prod_{d} a_{i,j}$ as the product of all expressions of the 
 form $\Pi_{B}$ for all propagating components of $d$. 
\end{definition}

\begin{example}
 For the partition diagram 
 $$ d = 
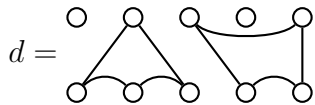
\begin{tikzpicture}[scale = 0.5,thick, baseline={(0,-1ex/2)}] 
\tikzstyle{vertex} = [shape = circle, minimum size = 7pt, inner sep = 1pt] 
\node[vertex] (G--5) at (6.0, -1) [shape = circle, draw] {}; 
\node[vertex] (G--4) at (4.5, -1) [shape = circle, draw] {}; 
\node[vertex] (G-3) at (3.0, 1) [shape = circle, draw] {}; 
\node[vertex] (G-5) at (6.0, 1) [shape = circle, draw] {}; 
\node[vertex] (G--3) at (3.0, -1) [shape = circle, draw] {}; 
\node[vertex] (G--2) at (1.5, -1) [shape = circle, draw] {}; 
\node[vertex] (G--1) at (0.0, -1) [shape = circle, draw] {}; 
\node[vertex] (G-2) at (1.5, 1) [shape = circle, draw] {}; 
\node[vertex] (G-1) at (0.0, 1) [shape = circle, draw] {}; 
\node[vertex] (G-4) at (4.5, 1) [shape = circle, draw] {}; 
\draw[] (G-3) .. controls +(0.6, -0.6) and +(-0.6, -0.6) .. (G-5); 
\draw[] (G-5) .. controls +(0, -1) and +(0, 1) .. (G--5); 
\draw[] (G--5) .. controls +(-0.5, 0.5) and +(0.5, 0.5) .. (G--4); 
\draw[] (G--4) .. controls +(-0.75, 1) and +(0.75, -1) .. (G-3); 
\draw[] (G-2) .. controls +(0.75, -1) and +(-0.75, 1) .. (G--3); 
\draw[] (G--3) .. controls +(-0.5, 0.5) and +(0.5, 0.5) .. (G--2); 
\draw[] (G--2) .. controls +(-0.5, 0.5) and +(0.5, 0.5) .. (G--1); 
\draw[] (G--1) .. controls +(0.75, 1) and +(-0.75, -1) .. (G-2); 
\end{tikzpicture} $$
 and for the $5 \times 5$ case of \eqref{generalnxn}, we find that $$ \prod_{d} a_{i, j} = \prod_{d} A = \left( a_{2,1} a_{2,2} a_{2, 3} 
 \right) \left( a_{3, 4} a_{3, 5} a_{5, 4} a_{5, 5} \right). $$ 
\end{example}

 Definition \ref{20230999828818177217707A7M1A} puts us in a position to offer a full definition
 for the concept of the recombinant of a matrix, as below. 

\begin{definition}\label{definitionrecombinant}
 We define the \emph{recombinant} of the square matrix in \eqref{generalnxn} so that 
\begin{equation}\label{202308qqqq22241000018A7M71A}
 \text{Rec}^{\lambda}(A) = \sum_{d \in P_{n}(r)} \chi_{P_{n}(r)}^{\lambda}(d) \prod_{d} a_{i, j}. 
\end{equation}
\end{definition}

 For example, an explicit evaluation for the recombinant, for non-propagating submodules of partition algebras, of any $2 \times 2$ matrix is given 
 Section \ref{subsectionexplicit} 

 Since our article is based on generalizing immanants using partition algebra characters, it would be appropriate to prove, as below, that Definition  
 \ref{definitionrecombinant} does indeed generalize \eqref{generalImm}.  In our below proof, we are to make use of the property described by Halverson 
 \cite{Halverson2001}  whereby character tables for partition algebras satisfy a recursion of the form 
\begin{equation}\label{Xirecursion}
 \Xi_{P_{n}(x)} = \left[ \begin{matrix} 
  x \Xi_{P_{n-1}(x)} & \vdots & \ast \\ 
 \cdots & \null & \cdots \\ 
 0 & \vdots & \Xi_{S_{n}} 
 \end{matrix} \right], 
\end{equation}
 where $ \Xi_{S_{n}} $ denotes the character table of $S_{n}$. 

 By direct analogy with how Young tableaux are formed from paths in Young's lattice, \emph{vacillating tableaux} are formed from paths in the Bratteli 
 diagram $\hat{A}$ described in \cite{HalversonRam2005}. For the case whereby such a path ends on an integer partition of order $n$ at level $n$ in $ 
 \hat{A}$, this corresponds to an embedding of an irreducible representation of $\mathbb{C}S_n$ \cite{HalversonRam2005}. For a vacillating tableau 
 $T$ of this form, Theorem \ref{maintheorem} below gives us that the recombinant corresponding to the partition algebra representation $\rho$ 
 corresponding to $T$ is the same as the immanant corresponding to the symmetric group algebra representation corresponding to $\rho$. 

\begin{theorem}\label{maintheorem}
 For an $n \times n$ matrix $A$, 
 if $|\lambda^{\ast}| = n$, then $\text{\emph{Rec}}^{\lambda}(A) = \text{\emph{Imm}}^{\lambda^{\ast}}(A)$. 
\end{theorem}

\begin{proof} 
 First, let us write $\lambda \vdash r$ and $|\lambda^{\ast}| \leq n$, and let us suppose that $\mu$ is a weak composition such that $0 \leq |\mu| \leq n$. 
 By Corollary 4.2.3 from \cite{Halverson2001}, we have that 
\begin{equation}\label{reducedcases}
 \chi_{P_{n}(r)}^{\lambda}\left( d_{\mu} \right) = 0 \ \ \ \text{if $|\mu| < |\lambda^{\ast}|$,} 
\end{equation}
 and that the equality $|\mu| = |\lambda^{\ast}| = n$ implies that 
\begin{equation}\label{characterequivalence}
 \chi_{P_{n}(r)}^{\lambda}\left( d_{\mu} \right) = \chi_{S_{n}} ^{\lambda^{\ast}}\left( \gamma_{\mu} \right). 
\end{equation}
 For a permuting diagram $d$, Halverson's procedure for conjugacy class analogues \cite{Halverson2001} gives us that $\gamma_{\mu}$ is the 
 cycle type for the permutation corresponding to $d$, with $d = d_{\mu}$ written as a product of disjoint, cyclic permutation diagrams. So, for an $n 
 \times n$ matrix $A$ and for $|\lambda^{\ast}| = n$, we find, from \eqref{reducedcases}, that $\chi_{P_{n}(r)}^{\lambda}(d)$ vanishes for all 
 non-propagating partition diagrams $d$, as in the lower left block of the character table in \eqref{Xirecursion}, so that we may rewrite 
 \eqref{202308qqqq22241000018A7M71A} so that 
\begin{equation}\label{Recpermutation}
 \text{Rec}^{\lambda}(A) = \sum_{\text{prop}(d) = n} \chi_{P_{n}(r)}^{\lambda}(d) \prod_{d} a_{i, j}, 
\end{equation}
 and where the character $ \chi_{P_{n}(r)}^{\lambda}(d) $ reduces, in the manner specified in \eqref{characterequivalence}, to the corresponding 
 character of $S_{n}$ evaluated at the permutation corresponding to the permuting diagram $d$. 
 By Definition \ref{20230999828818177217707A7M1A}, 
 the product $\prod_{d} a_{i, j}$ in \eqref{Recpermutation} is equal to 
 $a_{1, d(1)} a_{2, d(2)} \cdots a_{n, d(n)}$, writing the permuting diagram $d$ as a permutation 
 as in \eqref{pasfunction}. 
\end{proof}

\begin{remark}
 Let us write $E_{\ell}$ to denote the partition diagram corresponding to $$\frac{1}{r} \{ \{ 1, 1' \}, \{ 2, 2' \}, 
 \ldots, \{ \ell - 1, (\ell - 1)' \}, 
 \{ \ell, \ell + 1, \ldots, n \}, \{ \ell', (\ell + 1)', \ldots, n' \} \}.$$
 We find that $P_{n}(r) E_{\ell} P_{n}(r)$ 
 is a two-sided ideal and consists of all linear combinations of partition diagrams with propagation number strictly
 less than $\ell$. Fundamental results in the representation theory 
 of partition algebras are such that
\begin{equation}\label{symmetricquotient}
 \mathbb{C}S_n \cong P_{n}(r) / \left( P_{n}(r) E_{n} P_{n}(r) \right) 
 \end{equation} 
 and such that any irreducible representation of $P_{n}(r)$ is either an irreducible 
 representation of $E_{n} P_{n}(r) E_{n}$ or an irreducible representation of 
 the right-hand side of \eqref{symmetricquotient}; 
 see \cite[\S4]{Marcott2015}, for example, and references therein. 
 These properties can be used to formulate an alternative proof of Theorem \ref{maintheorem}. 
\end{remark}

 Our generalization of immanants, as above, is fundamentally different compared to previously considered generalizations or variants of the immanant 
 function. Notably, Definition \ref{definitionrecombinant} is separate relative to how \emph{$f$-immanants} are defined. Following 
 \cite{RhoadesSkandera2006}, an $f$-immanant, by analogy with \eqref{generalImm}, is of the form 
\begin{equation}\label{Immf}
 \text{Imm}^{f}(A) = \sum_{\sigma \in S_{n}} f(\sigma) \prod_{i=1}^{n} a_{i, \sigma_{i}} 
\end{equation}
 for an arbitrary function $f\colon S_{n} \to \mathbb{C}$. A notable instance of an $f$-immanant that is not of the form indicated in \eqref{generalImm} 
 is the Kazhdan--Lusztig immanant, where the $f$-function in \eqref{Immf} is given by Kazhdan--Lusztig polynomials associated to certain permutations. 
 In contrast to generalizations of immanants of the form shown in \eqref{Immf}, our lifting of the 
 definition in \eqref{generalImm} is based on a sum 
 indexed by the diagram basis of $P_{n}(r)$, in contrast to the index set for the sum in \eqref{Immf}.
 In contrast to immanants of $n \times n$ matrices being in correspondence with integer partitions of 
 $n$, and in contrast to $f$-immanants of $n \times n$ matrices being in correspondence with class functions on $S_{n}$, we have that recombinants of 
 $n \times n$ matrices are in correspondence with the family of integer partitions in \eqref{irreducibleindexed}. 

\subsection{An explicit evaluation}\label{subsectionexplicit}
 We find it convenient to denote partition algebra characters by writing $\chi^{\lambda^{\ast}}(d)$ in place of $\chi_{P_{n}(r)}^{\lambda}(d)$. 
 Correspondingly, we may denote the recombinant associated with the character $\chi^{\lambda^{\ast}}$ as $\text{Rec}^{\lambda^{\ast}}$. As 
 below, we are to let diagram basis elements be ordered according to the {\tt SageMath} convention for ordering such basis elements. According to this 
 convention, let the diagram basis of the order-$2$ partition diagram be ordered in the manner indicated in Table \ref{TablePCCE}, 
 letting partition diagrams be denoted with set partitions. 

\begin{table}[t]
\centering

 \begin{tabular}{ | c | c | c | }
 \hline
 $i$ & $d_{i}$ & $\chi^{\varnothing}(d_{i})$ \\ \hline 
 1 & $ \{\{2'$, $1'$, $1$, $ 2\}\}$ & $1$ \\ \hline 
 2 & $ \{\{2'$, $ 1$, $ 2\}$, $ \{1'\}\}$ & $1$ \\ \hline 
 3 & $ \{\{2'\}$, $ \{1'$, $1$, $ 2\}\}$ & $1$ \\ \hline 
 4 & $ \{\{2'$, $ 1'\}$, $ \{1$, $ 2\}\}$ & $r$ \\ \hline 
 5 & $ \{\{2'\}$, $ \{1'\}$, $ \{1$, $ 2\}\}$ & $r$ \\ \hline 
 6 & $ \{\{2'$, $ 1'$, $ 1\}$, $ \{2\}\}$ & $1$ \\ \hline 
 7 & $ \{\{2'$, $ 1\}$, $ \{1'$, $ 2\}\}$ & $2$ \\ \hline 
 8 & $ \{\{2'$, $ 1\}$, $ \{1'\}$, $ \{2\}\}$ & $r$ \\ \hline 
 9 & $ \{\{2'$, $2\}$, $ \{1'$, $ 1\}\}$ & $2$ \\ \hline 
 10 & $ \{\{2'$, $ 1'$, $ 2\}$, $ \{1\}\}$ & $1$ \\ \hline 
 11 & $ \{\{2'$, $ 2\}$, $ \{1'\}$, $ \{1\}\}$ & $r$ \\ \hline 
 12 & $ \{\{2'\}$, $ \{1'$, $ 1\}$, $ \{2\}\}$ & $r$ \\ \hline 
 13 & $ \{\{2'\}$, $ \{1'$, $ 2\}$, $ \{1\}\}$ & $r$ \\ \hline 
 14 & $ \{\{2'$, $ 1'\}$, $ \{1\}$, $ \{2\}\}$ & $r$ \\ \hline 
 15 & $ \{\{2'\}$, $ \{1'\}$, $ \{1\}$, $ \{2\}\}$ & $r^2$ \\ \hline 
 \end{tabular}

 \ 

\caption{The {\tt SageMath} ordering for partition diagrams of order $2$, along with 
 the irreducible characters corresponding to non-propagating representations. }\label{TablePCCE}
\end{table}

\begin{example}\label{exampleRecnull}
 According to Definition \ref{definitionrecombinant}, by writing 
\begin{align*}
 & \text{Rec}^{\varnothing}\left( \begin{matrix} 
 a_{1, 1} & a_{1, 2} \\ 
 a_{2, 1} & a_{2, 2} 
 \end{matrix}\right) = \sum_{d \in P_{2}(r)} \chi^{\varnothing}(d) \prod_{d} a_{i, j} \\ 
 & = \chi^{\varnothing}(d_{1}) \prod_{d_{1}} a_{i, j} + 
 \chi^{\varnothing}(d_{2}) \prod_{d_{2}} a_{i, j} + \cdots + 
 \chi^{\varnothing}(d_{15}) \prod_{d_{15}} a_{i, j}, 
\end{align*}
 we may evaluate the recombinant $\text{Rec}^{\varnothing}$ according to the character values shown in Table 
 \ref{TablePCCE}, so as to obtain that 
\begin{align*}
 & \text{Rec}^{\varnothing}\left( \begin{matrix} 
 a_{1, 1} & a_{1, 2} \\ 
 a_{2, 1} & a_{2, 2} 
 \end{matrix}\right) = \\ 
 & a_{1, 1} a_{1, 2} a_{2, 1} a_{2, 2} + \\ 
 & a_{1, 1} a_{2, 1} + a_{1, 1} a_{1, 2} + a_{1, 2} a_{2, 2} + a_{2, 1}a_{2, 2} + \\ 
 & 2 \left( a_{1, 1}a_{2, 2} + a_{1, 2} a_{2, 1} \right) + \\ 
 & r \left( a_{1, 1} + a_{1, 2} + a_{2, 1} + a_{2, 2} \right). 
\end{align*}
 For example, we may verify the above evaluation 
 by computing the traces associated with the linear transforms 
 given by the action of left-multiplication by diagram basis elements
 on the irreducible $P_{2}(r)$-module
 $\mathscr{L}\{ d_{4}, d_{14} \}$. 
\end{example}

 We may obtain a similar evaluation, relative to Example \ref{exampleRecnull}, for the recombinant that corresponds to the 3-dimensional representations 
 of $P_{2}(r)$. 

\section{Conclusion}
 We conclude with some areas for future research concerning the matrix function introduced in this paper. 

 A fundamental formula in algebraic combinatorics is Frobenius' formula for irreducible characters of the symmetric group, which, following 
 \cite{Halverson2001}, was later shown by Schur to be a consequence of what is now know as \emph{Schur--Weyl duality} between symmetric groups 
 and general linear groups. The \emph{irreducible character basis} introduced in \cite{OrellanaZabrocki2021} may be defined via a lifting of 
 the consequence 
\begin{equation}\label{consequenceSW}
 p_{\mu} = \sum_{\lambda \vdash n} \chi_{S_{n}}^{\lambda}(\mu) s_{\lambda} 
\end{equation} 
 of Schur--Weyl duality, with partition algebra characters used in place of symmetric group characters in an analogue of \eqref{consequenceSW}. 
 The {\tt SageMath} implementation of the $\tilde{s}$-basis from \cite{OrellanaZabrocki2021} provides a useful way of computing partition algebra 
 characters, which could be used to obtain a useful way of computing recombinants. We encourage applications of this. 

 Temperley--Lieb algebras form an important family of subalgebras of partition algebras. The \emph{Temperley--Lieb immanants} introduced by Rhoades 
 and Skandera \cite{RhoadesSkandera2005} are $f$-immanants defined in a way related to Temperley--Lieb algebras, referring to 
 \cite{RhoadesSkandera2005} for details. It seems that past research influenced by \cite{RhoadesSkandera2005}, including relevant research on 
 immanants or immanant-type functions as in \cite{deGuiseSpivakKulpDhand2016,Pylyavskyy2010,RhoadesSkandera2006,RhoadesSkandera2010}, has not 
 involved any generalizations of immanants using partition algebra characters. It may be worthwhile to explore relationships among recombinants 
 and Temperley--Lieb immanants, or to explore generalizations or variants of recombinants related to the way Temperley--Lieb immanants are defined. 

 The concept of a \emph{twisted immanant} was introduced in \cite{Itoh2016} and was based on how the irreducible character $\chi^{\lambda}$, if 
 restricted to an alternating subgroup, splits as a sum of two irreducible characters, writing $\chi^{\lambda} = \chi^{\lambda_{+}} + \chi^{\lambda_{-}}$. 
 What would be an appropriate notion of a \emph{twisted recombinant}, and how could this be applied in a similar way, relative to \cite{Itoh2016}? 

 Immanants are often applied in the field of algebraic graph theory, via immanants of Laplacian matrices and the like. How could recombinants be 
 applied similarly? 

 Immanants of Toeplitz matrices are often studied due to recursive properties of such immanants. What is the recombinant of a given Toeplitz matrix? 

\subsection*{Acknowledgements} 
 The author was supported through a Killam Postdoctoral Fellowship from the Killam Trusts, and the author wants to thank Karl Dilcher for many useful 
 discussions. The author is thankful to Mike Zabrocki for useful comments concerning the irreducible character basis and for many useful discussions 
 concerning partition algebras.

 \

  Department of Mathematics and Statistics

  Dalhousie University

  Halifax, NS,  B3H 4R2 

  {\tt jmaxwellcampbell@gmail.com}

\end{document}